\documentclass[reqno]{amsart}
\usepackage{amssymb,amsmath,hyperref}
\usepackage{amsrefs}
\usepackage[foot]{amsaddr}
\usepackage{bbold,stackrel}
 
\newtheorem{thm}{Theorem}

\newtheorem{cor}[thm]{Corollary}
\newtheorem{defi}[thm]{Definition}
\newtheorem{rem}[thm]{Remark}
\newtheorem{nota}[thm]{Notation}
\newtheorem{exa}[thm]{Example}

\newtheorem{tempie}[thm]{Template}
\newtheorem{ack}[thm]{Acknowledgement}

\newtheorem*{tempo*}{Template}

\newcommand\be{\begin{equation}}
\newcommand\ee{\end{equation}} 

\usepackage{amsmath,amsfonts} 

\usepackage[applemac]{inputenc}

\def\bdefi{\begin{defi}\rm}
\def\edefi{\end{defi}}
\def\bnota{\begin{nota}\rm}
\def\enota{\end{nota}}
\def\FIVE{\Pi_{1}^{1}\text{-\textup{\textsf{CA}}}_{0}}

\def\SIXK{\Pi_{k}^{1}\text{-\textsf{\textup{CA}}}_{0}^{\omega}}

\def\ATR{\textup{\textsf{ATR}}}

\def\W{\textup{\textsf{W}}}
\def\Y{\textup{\textsf{Y}}}

\def\Z{\textup{\textsf{Z}}}

\def\c{\textup{\textsf{c}}}
\def\RCA{\textup{\textsf{RCA}}}
\def\({\textup{(}}
\def\){\textup{)}}
\def\WO{\textup{\textsf{WO}}}
\def\RCAo{\textup{\textsf{RCA}}_{0}^{\omega}}
\def\ACAo{\textup{\textsf{ACA}}_{0}^{\omega}}
\def\WKL{\textup{\textsf{WKL}}}

\def\WWKL{\textup{\textsf{WWKL}}}
\def\bye{\end{document}}
\def\N{{\mathbb  N}}
\def\Q{{\mathbb  Q}}
\def\R{{\mathbb  R}}
\def\L{\textsf{\textup{L}}}

\def\MUC{\textup{\textsf{MUC}}}

\def\di{\rightarrow}
\def\asa{\leftrightarrow}
\def\ACA{\textup{\textsf{ACA}}}

\def\QFAC{\textup{\textsf{QF-AC}}}

\def\HBU{\textup{\textsf{HBU}}}

\def\LIN{\textup{\textsf{LIN}}}
\def\WHBU{\textup{\textsf{WHBU}}}

\def\UATR{\textup{\textsf{UATR}}}

\def\eps{\varepsilon}

\def\X{\textup{\textsf{X}}}

\def\FF{\textup{\textsf{FF}}}

\def\ECF{\textup{\textsf{ECF}}}

\newcommand{\T}{\textsf{\textup{T}}}

\usepackage{graphicx}
\usepackage{tikz}
\usetikzlibrary{matrix, shapes.misc}

\numberwithin{equation}{section}
\numberwithin{thm}{section}

\usepackage{comment}

\begin{document}
\title{Splittings and disjunctions in Reverse Mathematics}
\author{Sam Sanders}
\address{School of Mathematics, University of Leeds \& Dept.\ of Mathematics, TU Darmstadt}
\email{sasander@me.com}

\begin{abstract}
Reverse Mathematics (RM hereafter) is a program in the foundations of mathematics founded by Friedman and developed extensively by Simpson and others.  
The aim of RM is to find the minimal axioms needed to prove a theorem of ordinary, i.e.\ non-set-theoretic, mathematics.  As suggested by the title, 
this paper deals with two (relatively rare) RM-phenomena, namely \emph{splittings} and \emph{disjunctions}.  
As to splittings, there are some examples in RM of theorems $A, B, C$ such that $A\asa (B\wedge C)$, i.e.\ $A$ can be \emph{split} into two independent (fairly natural) parts $B$ and $C$.  
As to disjunctions, there are (very few) examples in RM of theorems $D, E, F$ such that $D\asa (E\vee F)$, i.e.\ $D$ can be written as the \emph{disjunction} of two independent (fairly natural) parts $E$ and $F$.  
By contrast, we show in this paper that there is a plethora of (natural) splittings and disjunctions in Kohlenbach's \emph{higher-order} RM.  
\end{abstract}


\maketitle
\thispagestyle{empty}

\vspace{-0.5cm}
\section{Introduction}\label{intro}
Reverse Mathematics (RM hereafter) is a program in the foundations of mathematics initiated around 1975 by Friedman (\cites{fried,fried2}) and developed extensively by Simpson (\cite{simpson2}) and others.  
We refer to \cite{stillebron} for a basic introduction to RM and to \cite{simpson2, simpson1} for an (updated) overview of RM.  We will assume basic familiarity with RM, the associated `Big Five' systems and the `RM zoo' (\cite{damirzoo}).  
We do introduce Kohlenbach's \emph{higher-order} RM in some detail Section \ref{HORM}.    

\smallskip

As discussed in e.g.\ \cite{dsliceke}*{\S6.4}, there are (some) theorems $A, B, C$ in the RM zoo such that $A\asa (B\wedge C)$, i.e.\ $A$ can be \emph{split} into two independent (fairly natural) parts $B$ and $C$ (over $\RCA_{0}$).  
As to the possibility of $A\asa (B\vee C)$, there is \cite{yukebox}*{Theorem~4.5} which states that a certain theorem about dynamical systems is equivalent to the \emph{disjunction} of weak K\"onig's lemma and induction for $\Sigma_{2}^{0}$-formulas; neither disjunct of course implies the other (over $\RCA_{0}$).
Similar results are in \cite{boulanger} for model theory, but these are more logical in nature.     

\smallskip

It is fair to say that there are only few \emph{natural} examples of splittings and disjunctions in RM, though such claims are invariably subjective in nature. 
Nonetheless, the aim of this paper is to establish a \emph{plethora} of splittings and disjunctions in {higher-order} RM.  In particular, we obtain splittings and disjunctions involving (higher-order) $\WWKL_{0}$, the Big Five, and $\Z_{2}$, among others.  We similarly treat the covering theorems \emph{Cousin's lemma} and \emph{Lindel\"of's lemma} studied in \cite{dagsamIII}.  
Our main results are in Section \ref{main}, while a summary may be found in Section \ref{konkelfoes}; our base theories are generally conservative over $\WKL_{0}$ (or are strictly weaker). 

\smallskip

It goes without saying that our results highlight a \emph{major} difference between second- and higher-order arithmetic, and the associated development of RM.  
We provide some musings on this and related foundational matters in Section \ref{opinion}.  
\section{Preliminaries}
\subsection{Higher-order Reverse Mathematics}\label{HORM}
We sketch Kohlenbach's \emph{higher-order Reverse Mathematics} as introduced in \cite{kohlenbach2}.  In contrast to `classical' RM, higher-order RM makes use of the much richer language of \emph{higher-order arithmetic}.  

\smallskip

As suggested by its name, {higher-order arithmetic} extends second-order arithmetic.  Indeed, while the latter is restricted to numbers and sets of numbers, higher-order arithmetic also has sets of sets of numbers, sets of sets of sets of numbers, et cetera.  
To formalise this idea, we introduce the collection of \emph{all finite types} $\mathbf{T}$, defined by the two clauses:
\begin{center}
(i) $0\in \mathbf{T}$   and   (ii)  If $\sigma, \tau\in \mathbf{T}$ then $( \sigma \di \tau) \in \mathbf{T}$,
\end{center}
where $0$ is the type of natural numbers, and $\sigma\di \tau$ is the type of mappings from objects of type $\sigma$ to objects of type $\tau$.
In this way, $1\equiv 0\di 0$ is the type of functions from numbers to numbers, and where  $n+1\equiv n\di 0$.  Viewing sets as given by characteristic functions, we note that $\Z_{2}$ only includes objects of type $0$ and $1$; we denote the associated language by $\L_{2}$.

\smallskip

The language $\L_{\omega}$ includes variables $x^{\rho}, y^{\rho}, z^{\rho},\dots$ of any finite type $\rho\in \mathbf{T}$.  Types may be omitted when they can be inferred from context.  
The constants of $\L_{\omega}$ includes the type $0$ objects $0, 1$ and $ <_{0}, +_{0}, \times_{0},=_{0}$  which are intended to have their usual meaning as operations on $\N$.
Equality at higher types is defined in terms of `$=_{0}$' as follows: for any objects $x^{\tau}, y^{\tau}$, we have
\be\label{aparth}
[x=_{\tau}y] \equiv (\forall z_{1}^{\tau_{1}}\dots z_{k}^{\tau_{k}})[xz_{1}\dots z_{k}=_{0}yz_{1}\dots z_{k}],
\ee
if the type $\tau$ is composed as $\tau\equiv(\tau_{1}\di \dots\di \tau_{k}\di 0)$.  
Furthermore, $\L_{\omega}$ also includes the \emph{recursor constant} $\mathbf{R}_{\sigma}$ for any $\sigma\in \mathbf{T}$, which allows for iteration on type $\sigma$-objects as in the special case \eqref{special}.  
Formulas and terms are defined as usual.  
\bdefi The base theory $\RCAo$ consists of the following axioms:
\begin{enumerate}
 \renewcommand{\theenumi}{\alph{enumi}}
\item  Basic axioms expressing that $0, 1, <_{0}, +_{0}, \times_{0}$ form an ordered semi-ring with equality $=_{0}$.
\item Basic axioms defining the well-known $\Pi$ and $\Sigma$ combinators (aka $K$ and $S$ in \cite{avi2}), which allow for the definition of \emph{$\lambda$-abstraction}. 
\item The defining axiom of the recursor constant $\mathbf{R}_{0}$: For $m^{0}$ and $f^{1}$: 
\be\label{special}
\mathbf{R}_{0}(f, m, 0):= m \textup{ and } \mathbf{R}_{0}(f, m, n+1):= f( \mathbf{R}_{0}(f, m, n)).
\ee
\item The \emph{axiom of extensionality}: for all $\rho, \tau\in \mathbf{T}$, we have:
\be\label{EXT}\tag{$\textsf{\textup{E}}_{\rho, \tau}$}  
(\forall  x^{\rho},y^{\rho}, \varphi^{\rho\di \tau}) \big[x=_{\rho} y \di \varphi(x)=_{\tau}\varphi(y)   \big].
\ee 
\item The induction axiom for quantifier-free\footnote{To be absolutely clear, variables (of any finite type) are allowed in quantifier-free formulas of the language $\L_{\omega}$: only quantifiers are banned.} formulas of $\L_{\omega}$.
\item $\QFAC^{1,0}$: The quantifier-free axiom of choice as in Definition \ref{QFAC}.
\end{enumerate}
\edefi
\bdefi\label{QFAC} The axiom $\QFAC$ consists of the following for all $\sigma, \tau \in \textbf{T}$:
\be\tag{$\QFAC^{\sigma,\tau}$}
(\forall x^{\sigma})(\exists y^{\tau})A(x, y)\di (\exists Y^{\sigma\di \tau})(\forall x^{\sigma})A(x, Y(x)),
\ee
for any quantifier-free formula $A$ in the language of $\L_{\omega}$.
\edefi
As discussed in \cite{kohlenbach2}*{\S2}, $\RCAo$ and $\RCA_{0}$ prove the same sentences `up to language' as the latter is set-based and the former function-based.  Recursion as in \eqref{special} is called \emph{primitive recursion}; the class of functionals obtained from $\mathbf{R}_{\rho}$ for all $\rho \in \mathbf{T}$ is called \emph{G\"odel's system $T$} of all (higher-order) primitive recursive functionals.  

\smallskip

We use the usual notations for natural, rational, and real numbers, and the associated functions, as introduced in \cite{kohlenbach2}*{p.\ 288-289}.  
\begin{defi}[Real numbers and related notions in $\RCAo$]\label{keepintireal}\rm~
\begin{enumerate}
 \renewcommand{\theenumi}{\alph{enumi}}
\item Natural numbers correspond to type zero objects, and we use `$n^{0}$' and `$n\in \N$' interchangeably.  Rational numbers are defined as signed quotients of natural numbers, and `$q\in \Q$' and `$<_{\Q}$' have their usual meaning.    
\item Real numbers are represented by fast-converging Cauchy sequences $q_{(\cdot)}:\N\di \Q$, i.e.\  such that $(\forall n^{0}, i^{0})(|q_{n}-q_{n+i})|<_{\Q} \frac{1}{2^{n}})$.  
We use the `hat function' from \cite{kohlenbach2}*{p.\ 289} to guarantee that any $f^{1}$ defines a real number.  
\item We write `$x\in \R$' to express that $x^{1}:=(q^{1}_{(\cdot)})$ represents a real as in the previous item and write $[x](k):=q_{k}$ for the $k$-th approximation of $x$.    
\item Two reals $x, y$ represented by $q_{(\cdot)}$ and $r_{(\cdot)}$ are \emph{equal}, denoted $x=_{\R}y$, if $(\forall n^{0})(|q_{n}-r_{n}|\leq \frac{1}{2^{n-1}})$. Inequality `$<_{\R}$' is defined similarly.  
We sometimes omit the subscript `$\R$' if it is clear from context.           
\item Functions $F:\R\di \R$ are represented by $\Phi^{1\di 1}$ mapping equal reals to equal reals, i.e. $(\forall x , y\in \R)(x=_{\R}y\di \Phi(x)=_{\R}\Phi(y))$.
\item The relation `$x\leq_{\tau}y$' is defined as in \eqref{aparth} but with `$\leq_{0}$' instead of `$=_{0}$'.  Binary sequences are denoted `$f^{1}, g^{1}\leq_{1}1$', but also `$f,g\in C$' or `$f, g\in 2^{\N}$'.  
\end{enumerate}
\end{defi}
We now discuss the issue of representations of real numbers.
\begin{rem}\label{forealsteve}\rm
Introductory analysis courses often provide an explicit construction of $\R$ (perhaps in an appendix), while in practice one generally makes use of the axiomatic properties of $\R$, and not the explicit construction.  
Now, there are a number of different\footnote{The `early' constructions due to Dedekind (see e.g.\ \cite{kindke}; using cuts) and Cantor (see e.g.\ \cite{cant}; using Cauchy sequences) were both originally published in 1872.} such constructions: Tao uses Cauchy sequences in his text \cite{taoana1} and discusses decimal expansions in the Appendix \cite{taoana1}*{\S B}.  Hewitt-Stromberg also use Cauchy sequences in \cite{hestrong}*{\S5} and discuss Dedekind cuts in the exercises (\cite{hestrong}*{p.~46}).  Rudin uses Dedekind cuts in \cite{rudin} and mentions that Cauchy sequences yield the same result. 
Clearly, Definition \ref{keepintireal} is based on Cauchy sequences, but Hirst has shown that over $\RCA_{0}$, individual real numbers can be converted between various representations (\cite{polahirst}).  Thus, the choice of representation in Definition \ref{keepintireal} does not really matter, even over $\RCA_{0}$.  
Moreover, the latter proves (\cite{simpson2}*{II.4.5}) that the real number system satisfies all the axioms of an Archimedian ordered field, i.e.\ we generally work with the latter axiomatic properties in RM, rather than with the representations (whatever they are).     
\end{rem}
Finally, we mention the $\ECF$-interpretation, as it will be needed below.
\begin{rem}\label{ECFrem}\rm
 The technical definition of the $\ECF$-interpretation may be found in \cite{troelstra1}*{p.\ 138, 2.6}.
Intuitively speaking, the $\ECF$-interpretation $[A]_{\ECF}$ of a formula $A\in \L_{\omega}$ is just $A$ with all variables 
of type two and higher replaced by countable representations of continuous functionals. 
The $\ECF$-interpretation connects $\RCAo$ and $\RCA_{0}$ (See \cite{kohlenbach2}*{Prop.\ 3.1}) in that if $\RCAo$ proves $A$, then $\RCA_{0}$ proves $[A]_{\ECF}$, again `up to language', as $\RCA_{0}$ is 
formulated using sets, and $[A]_{\ECF}$ is formulated using types, namely only using type zero and one objects.  Note that for $A\in \L_{2}$, we have that $[A]_{\ECF}$ is just $A$ by definition.   
\end{rem}
\subsection{Some axioms of higher-order arithmetic}\label{axies!}
We introduce some functionals which constitute the counterparts of $\Z_{2}$, and some of the Big Five systems, in higher-order RM.
We use the formulation of these functionals as in \cite{kohlenbach2}.  

\smallskip
\noindent
First of all, $\ACA_{0}$ is readily derived from the following `Turing jump' functional:
\be\label{muk}\tag{$\exists^{2}$}
(\exists \varphi^{2}\leq_{2}1)(\forall f^{1})\big[(\exists n)(f(n)=0) \asa \varphi(f)=0    \big]. 
\ee
and $\ACA_{0}^{\omega}\equiv\RCAo+(\exists^{2})$ proves the same sentences as $\ACA_{0}$ by \cite{hunterphd}*{Theorem~2.5}. 
This functional is \emph{discontinuous} at $f=_{1}11\dots$, and $(\exists^{2})$ is equivalent to the existence of $F:\R\di\R$ such that $F(x)=1$ if $x>_{\R}0$, and $0$ otherwise (\cite{kohlenbach2}*{\S3}).  

\smallskip
\noindent
Secondly, $\FIVE$ is readily derived from the following `Suslin functional':
\be\tag{$S^{2}$}
(\exists S^{2}\leq_{2}1)(\forall f^{1})\big[  (\exists g^{1})(\forall x^{0})(f(\overline{g}n)=0)\asa S(f)=0  \big], 
\ee
and $\FIVE^{\omega}\equiv \RCAo+(S^{2})$ proves the same $\Pi_{3}^{1}$-sentences as $\FIVE$ by \cite{yamayamaharehare}*{Theorem 2.2}.   
By definition, the Suslin functional $S^{2}$ can decide whether a $\Sigma_{1}^{1}$-formula (as in the left-hand side of $(S^{2})$) is true or false.   
Note that we allow formulas with (type one) \emph{function} parameters, but \textbf{not} with (higher type) \emph{functional} parameters.  The system $\SIXK$ is defined similarly via a functional $S_{k}^{2}$ deciding $\Sigma_{k}^{1}$-formulas.  

\smallskip
\noindent
Thirdly, full second-order arithmetic $\Z_{2}$ is readily derived from the sentence:
\be\tag{$\exists^{3}$}
(\exists E^{3}\leq_{3}1)(\forall Y^{2})\big[  (\exists f^{1})Y(f)=0\asa E(Y)=0  \big], 
\ee
and we define $\Z_{2}^{\Omega}\equiv \RCAo+(\exists^{3})$, a conservative extension of $\Z_{2}$ by \cite{hunterphd}*{Cor. 2.6}.   The (unique) functional from $(\exists^{3})$ is also called `$\exists^{3}$', and we will use a similar convention for other functionals.  

\smallskip
\noindent
Fourth, weak K\"onig's lemma\footnote{Note that we take `$\WKL$' to be the $\L_{2}$-sentence \emph{every infinite binary tree has a path} as in \cite{simpson2}, while the Big Five system $\WKL_{0}$ is $\RCA_{0}+\WKL$, and $\WKL_{0}^{\omega}$ is $\RCAo+\WKL$.} ($\WKL$ hereafter) easily follows from both the `intuitionistic' and `classical' \emph{fan functional}, 
which are defined as follows:  
\be\tag{$\MUC$}
(\exists \Omega^{3})(\forall Y^{2})(\forall f, g\in C)(\overline{f}\Omega(Y)=\overline{g}\Omega(Y)\di Y(f)=Y(g)),
\ee 
\be\tag{$\textsf{\textup{FF}}$}\label{FF}
(\exists \Phi^{3})(\forall Y^{2}\in \textsf{\textup{cont}})(\forall f, g\in C)(\overline{f}\Phi(Y)=\overline{g}\Phi(Y)\di Y(f)=Y(g)),
\ee
where `$Y^{2}\in \textsf{cont}$' means that $Y$ is continuous on Baire space $\N^{\N}$.  
Clearly, $\exists^{2}$, $S^{2}$, and $\exists^{3}$ are a kind of comprehension axiom.    
As it turns out, the \emph{comprehension for Cantor space} functional also yields a conservative extension of $\WKL_{0}$:
\be\tag{$\kappa_{0}^{3}$}
(\exists \kappa_{0}^{3})(\forall Y^{2})\big[ \kappa_{0}(Y)=0\asa (\exists f\in C)(Y(f)>0)  \big], 
\ee
as $\MUC$ implies $(\kappa_{0}^{3})$, and the former is conservative over $\WKL_{0}$ by \cite{kohlenbach2}*{Cor.\ 3.15}.
The subscript `0' in $(\kappa_{0}^{3})$ has no purpose other than distinguishing this axiom from the related axiom $(\kappa^{3})$ from \cite{dagsam}.

\smallskip

\noindent
Finally, recall that the Heine-Borel theorem (aka \emph{Cousin's lemma}) states the existence of a finite sub-cover for an open cover of a compact space. 
Now, a functional $\Psi:\R\di \R^{+}$ gives rise to the \emph{canonical} cover $\cup_{x\in I} I_{x}^{\Psi}$ for $I\equiv [0,1]$, where $I_{x}^{\Psi}$ is the open interval $(x-\Psi(x), x+\Psi(x))$.  
Hence, the uncountable cover $\cup_{x\in I} I_{x}^{\Psi}$ has a finite sub-cover by the Heine-Borel theorem; in symbols:
\be\tag{$\HBU$}
(\forall \Psi:\R\di \R^{+})(\exists \langle y_{1}, \dots, y_{k}\rangle){(\forall x\in I)}(\exists i\leq k)(x\in I_{y_{i}}^{\Psi}).
\ee
By the results in \cite{dagsamIII, dagsamV}, $\Z_{2}^{\Omega}$ proves $\HBU$, but $\Z_{2}^{\omega}\equiv \cup_{k}\SIXK$ cannot.  
The importance and naturalness of $\HBU$ is discussed in Section \ref{opinion}. 

\smallskip

Furthermore, since Cantor space (denoted $C$ or $2^{\N}$) is homeomorphic to a closed subset of $[0,1]$, the former inherits the same property.  
In particular, for any $G^{2}$, the corresponding `canonical cover' of $2^{\N}$ is $\cup_{f\in 2^{\N}}[\overline{f}G(f)]$ where $[\sigma^{0^{*}}]$ is the set of all binary extensions of $\sigma$.  By compactness, there is a finite sequence $\langle f_0 , \ldots , f_n\rangle$ such that the set of $\cup_{i\leq n}[\bar f_{i} F(f_i)]$ still covers $2^{\N}$.  By \cite{dagsamIII}*{Theorem 3.3}, $\HBU$ is equivalent to the same compactness property for $C$, as follows:
\be\tag{$\HBU_{\c}$}
(\forall G^{2})(\exists \langle f_{1}, \dots, f_{k} \rangle ){(\forall f^{1}\leq_{1}1)}(\exists i\leq k)(f\in [\overline{f_{i}}G(f_{i})]).
\ee
Note that $\MUC$ implies $\HBU_{\c}$, i.e.\ the latter has weak first-order strength, but is extremely hard to prove by the aforementioned results.  

\smallskip

Finally, we need a `trivially uniform' version of $\ATR_{0}$:
\be\tag{$\UATR$}
(\exists \Phi^{1\di 1})(\forall X^{1}, f^{1})\big[\WO(X)\di H_{f}(X, \Phi(X,f)) \big], 
\ee
where $\WO(X)$ expresses that $X$ is a countable well-ordering and $H_{\theta}(X, Y)$ expresses that $Y$ is the result from iterating $\theta$ along $X$ (See \cite{simpson2}*{V} for details), and where $H_{f}(X, Y)$ is just $H_{\theta}(X, Y)$ with $\theta(n, Z)$ defined as $(\exists m^{0})(f(n,m, \overline{Z}m)=0)$.

\section{Main results}\label{main}
Our motivation and starting point is the splitting $(\exists^{3})\asa [(\kappa_{0}^{3})\wedge(\exists^{2})]$ communicated to us by Kohlenbach\footnote{The proof amounts to the observation that $\N^\N$ is recursively homeomorphic to a $\Pi^0_2$-subset of Cantor space. Since this set is computable in $\exists^{2}$, any oracle call to $\exists^{3}$ can be rewritten to an equivalent oracle call to $\kappa_{0}^{3}$, in a uniform way.} (See \cite{dagsam}*{Rem.\ 6.13}). 
It is then a natural question if $(\kappa_{0}^{3})$ can be split further, as discussed in Section \ref{CoC}.  We obtain similar results for $\MUC$ in Section \ref{KUM}, which yields splittings and disjunctions for $(\exists^{2})$, $(\exists^{3})$, $(Z^{3})$, and $\FF$ in Section \ref{more}.
We similarly study $\HBU$ in Section \ref{disjunkie}, while other covering theorems, including the original \emph{Lindel\"of lemma}, are discussed in Section \ref{others}. 
As done in e.g.\ \cite{dsliceke}, we shall always write `$A+B$' in the stead of `$A\wedge B$'.
\subsection{Comprehension on Cantor space}\label{CoC}
We show that $(\kappa_{0}^{3})$ defined as follows:
\be\tag{$\kappa_{0}^{3}$}
(\exists \kappa_{0}^{3})(\forall Y^{2})\big[ \kappa_{0}(Y)=0\asa (\exists f\in C)(Y(f)>0)  \big], 
\ee
splits into the classical fan functional, given by $\FF$ as follows:
\be\tag{$\textsf{\textup{FF}}$}\label{FF2}
(\exists \Phi^{3})(\forall Y^{2}\in \textsf{\textup{cont}})(\forall f, g\in C)(\overline{f}\Phi(Y)=\overline{g}\Phi(Y)\di Y(f)=Y(g)),
\ee
and a functional which tests for continuity on $\N^{\N}$, as follows:
\be\tag{$Z^{3}$}
(\exists Z^{3})(\forall Y^{2})\big[Z(Y)=0\asa (\forall f^{1})(\exists N^{0})(\forall g^{1})(\overline{f}N=\overline{g}N\di Y(f)=Y(g))  \big].
\ee
We will tacitly use $(\exists^{2})\di \FF\di \WKL$, which holds over $\RCAo$ by \cite{kohlenbach4}*{Prop.\ 4.10}.  
\begin{thm}
The system $\WKL_{0}^{\omega}+\QFAC^{2,0}$ proves $(\kappa_{0}^{3})\asa \big[(Z^{3})+\FF\big]$.
\end{thm}
\begin{proof}
For the forward implication, we work in $\WKL_{0}^{\omega}+\QFAC^{2,0}+(\kappa_{0}^{3})$. In case $(\exists^{2})$ holds, we also have $(\exists^{3})$, and the latter functional readily implies $(Z^{3})$ and $\FF$. 
In case of $\neg(\exists^{2})$, all functionals $Y^{2}$ are continuous on Baire space by \cite{kohlenbach2}*{Prop.\ 3.7}, and $Z_{0}=_{3}0$ is as required for $(Z^{3})$.
By $\WKL$ (and \cite{kohlenbach4}*{Prop.\ 4.10}), all functionals $Y^{2}$ are uniformly continuous on Cantor space, i.e.\
\[
(\forall Y^{2})(\exists N^{0})\underline{(\forall f^{1}, g^{1}\in C)(\overline{f}N=\overline{g}N\di Y(f)=Y(g))}, 
\]
and the underlined formula may be treated as quantifier-free by $(\kappa_{0}^{3})$.  Applying $\QFAC^{2, 0}$, we obtain $\FF$.  The law of excluded middle finishes this implication.  

\smallskip

For the reverse implication, we work in $\RCAo+(Z^{3})+\FF$.  In case of $\neg(\exists^{2})$, all functionals $Y^{2}$ are continuous on Baire space by \cite{kohlenbach2}*{Prop.\ 3.7}, and $\FF$ readily implies $(\kappa_{0}^{3})$ by noting that the latter
restricted to $Y^{2}$ uniformly continuous on $C$ is trivial.  
In case of $(\exists^{2})$, let $Y_{0}$ be $Y$ on $C$, and zero otherwise.  
Now define $\kappa_{0}$ as follows: in case $Z(Y_{0})=0$, $Y_{0}$ is continuous on Cantor space, and use $\FF$ to decide whether $(\exists f\in C)(Y(f)>0)$; in case $Z(Y_{0})\ne 0$, then $(\exists f\in C)(Y_{0}(f)>0)$, and $\kappa_{0}(Y):=0$.   
The law of excluded middle finishes this implication.  
\end{proof}
\begin{cor}\label{doors}
The system $\RCAo+\QFAC^{2,0}$ proves $[(\kappa_{0}^{3})+\WKL]\asa [(Z^{3})+\FF]$ and the system $\RCAo$ proves $(\exists^{3})\asa \big[(\exists^{2})+(Z^{3})\big]$.
\end{cor}

\subsection{The intuitionistic fan functional}\label{KUM}
A hallmark of intuitionistic mathematics is \emph{Brouwer's continuity theorem} which expresses that all functions on the unit interval are (uniformly) continuous (\cite{brouw}).  
In the same vein, the \emph{intuitionistic fan functional} $\Omega^{3}$ as in $\MUC$ provides a \emph{modulus} of uniform continuity on Cantor space:
\be\tag{$\MUC$}
(\exists \Omega^{3})(\forall Y^{2})(\forall f, g\in C)(\overline{f}\Omega(Y)=\overline{g}\Omega(Y)\di Y(f)=Y(g)).
\ee 
This axiom can be split nicely into classical and non-classical parts as follows.
\begin{thm}\label{dorkesss}
The system $\RCAo+\QFAC^{2,0}$ proves 
\[
\MUC\asa [(\kappa_{0}^{3})+\WKL+\neg(\exists^{2})]\asa  [(\kappa_{0}^{3})+\WKL+\neg(S^{2})] \asa [(\kappa_{0}^{3})+\WKL+\neg(\exists^{3})].
\]
\end{thm}
\begin{proof}
For the first equivalence, assume $\MUC$ and note that the latter reduces the decision procedure for $(\exists f\in C)(Y(f)>0)$ to a finite search involving only $2^{\Omega(Y)}$ sequences.  
Furthermore, $(\exists^{2})$ clearly implies the existence of a discontinuous function on Cantor space, i.e.\ $\MUC\di \neg(\exists^{2})$ follows, while $\MUC\di \WKL$ follows from \cite{simpson2}*{IV.2.3}.  Now assume $(\kappa_{0}^{3})+\WKL+\neg(\exists^{2})$ and recall that by the latter all functionals $Y^{2}$ are continuous on Baire space by \cite{kohlenbach2}*{Prop.\ 3.7}.
By $\WKL$ (and \cite{kohlenbach4}*{Prop.\ 4.10}), all functionals $Y^{2}$ are uniformly continuous on Cantor space, i.e.\
\be\label{kortrerokseks}
(\forall Y^{2})(\exists N^{0})\underline{(\forall f^{1}, g^{1}\in C)(\overline{f}N=\overline{g}N\di Y(f)=Y(g))}, 
\ee
and the underlined formula may be treated as quantifier-free by $(\kappa_{0}^{3})$.  Applying $\QFAC^{2, 0}$, we obtain $\MUC$.  For the remaining equivalences, 
since $(\exists^{3})\asa [(\kappa^{3}_{0})+(\exists^{2})  ]$, we have that $[\neg(\exists^{3})+(\kappa_{0}^{3})]\di \neg(\exists^{2})$, and the same for $\neg(S^{2})$. 
Finally, note that $(\exists^{3})\di (S^{2})\di (\exists^{2})$ implies $\neg(\exists^{2})\di \neg(S^{2})\di\neg (\exists^{3})$.
\end{proof}
Recall the $\ECF$-interpretation introduced at the end of Section \ref{HORM}. 
By \cite{longmann}*{\S9.5}, we have $[\MUC]_{\ECF}\asa \WKL$ and $\WKL\di [(\kappa_{0}^{3})]_{\ECF}$, while $[(\exists^{2})]_{\ECF}\asa( 0=1)$ as $\exists^{2}$ is discontinuous (and therefore has no countable representation).  
Hence, $\neg(\exists^{2})$ \emph{cannot} be replaced by $\neg\ACA_{0}$ in the theorem, as $[A]_{\ECF}\asa A$ for $A\in \L_{2}$. 

\smallskip

Furthermore, the axiom $\MUC$ can \emph{also} be split as follows.  
As an exercise, the reader should show that the corollary also goes through for $\RCAo$.
\begin{cor}\label{conseuq}
The system $\RCAo+\QFAC^{2,0}$ proves 
\[
\MUC\asa [\FF+\neg(\exists^{2})]\asa [\FF+(Z^{3})+\neg(S^{2})]\asa  [\FF+(Z^{3})+\neg(\exists^{3})].
\]
\end{cor}
\begin{proof}
By Corollary \ref{doors} and the theorem, we have $\MUC\asa [ (Z^{3})+\FF + \neg(\exists^{2}) ]$, and we may omit $(Z^{3})$ because all functionals on $\N^{\N}$ are continuous given $\neg(\exists^{2})$.
By the same corollary, $ [(Z^{3})+\FF +\neg(S^{2})]\asa [(\kappa_{0}^{3})+\WKL+\neg(S^{2})]$, and the latter is equivalent to $\MUC$ by the theorem.  
The same reasoning applies to $\neg(\exists^{3})$.
\end{proof}
\noindent
%
%
%
As a result of the previous, the RM of $(\kappa_{0}^{3})$ is pretty robust.  Indeed, for a sentence $\W$ implying $(\kappa_{0}^{3})$, if the former implies the existence of a discontinuous functional, we obtain $(\exists^{3})$ by \cite{kohlenbach2}*{\S3}.  
What happens when $\W$ does \emph{not} imply this existence, is captured (in part) by the following theorem.
\begin{thm}
If $\MUC\di \W$ and $(\exists^{3})\di \W\di (\kappa_{0}^{3})$ over $\RCAo$, then $\WKL_{0}^{\omega}+\QFAC^{2,0}$ proves $\W\asa (\kappa_{0}^{3})$.
\end{thm}
\begin{proof}
The forward implication is immediate.  For the reverse implication, consider $(\exists^{2})\vee \neg(\exists^{2})$; in the former case, we obtain $(\exists^{3})$ and hence $\W$, while in the latter case, we may use the proof of Theorem \ref{dorkesss}: the continuity of all functionals on Baire space and $\WKL$ imply \eqref{kortrerokseks}, which yields $\MUC$ thanks to $(\kappa_{0}^{3})$ and $\QFAC^{2,0}$, and $\W$ follows by assumption.  
\end{proof}

\subsection{More splittings and disjunctions}\label{more}
The results regarding the non-classical axiom $\MUC$ also yield splittings for the classical axioms $\FF$, $(\exists^{2})$, $(\exists^{3})$, and $(Z^{3})$. 
\begin{thm}\label{builurhouse}
The system $\RCAo+\QFAC^{2,0}$ proves 
\be\label{bayor}
[(\kappa_{0}^{3})+\WKL ]\asa [(\exists^{3})\vee \MUC]\asa [(\kappa_{0}^{3})+\FF ];
\ee
$\RCAo$ proves $\FF\asa [(\exists^{2})\vee \MUC]$, while $\WKL_{0}^{\omega}+\QFAC^{2,0}$ proves $\FF \asa [(\exists^{2})\vee (\kappa_{0}^{3})]$.  
\end{thm}
\begin{proof}
For the first equivalence in \eqref{bayor}, the reverse implication is immediate if $(\exists^{3})$ holds, while it follows from Theorem \ref{dorkesss} if $\MUC$ holds.  For the forward implication, if $(\exists^{2})$, we have $(\exists^{3})$, while if $\neg(\exists^{2})$, we follow the proof of Theorem \ref{dorkesss} to obtain $\MUC$. 
The second equivalence in \eqref{bayor} follows in the same way. 
For the third equivalence, the reverse implication is immediate, while the forward implication follows by considering $(\exists^{2})\vee \neg(\exists^{2})$, noting that all functionals on $C$ are continuous in the latter case.  
For the final equivalence, we only need to prove $(\kappa_{0}^{3})\di \FF$ given $\WKL$.  The implication is immediate if $(\exists^{2})$, while it follows in the same way as in the proof of \eqref{bayor} in case $\neg(\exists^{2})$.  
\end{proof}
\begin{thm}\label{kwazoeloe}
The system $\RCAo+\QFAC^{2,0}$ proves $(\exists^{2})\asa [\FF+ \neg\MUC]$ and $(\exists^{3})\asa [\FF+(Z^{3})+\neg\MUC]$ and $(Z^{3})\asa [(\exists^{3}) \vee \neg(\exists^{2})]\asa [(\exists^{3}) \vee \neg\FF\vee \MUC ]$. 
\end{thm}
\begin{proof}
The second equivalence follows from the first one by Corollary \ref{doors}.
For the first equivalence, the forward implication is immediate, and for the reverse implication, Corollary \ref{conseuq} implies $\neg\MUC\asa [\neg \FF\vee (\exists^{2})]$.  
Since $\FF$ is assumed, we obtain $(\exists^{2})$.  For the third equivalence, the reverse implication is immediate in case $(\exists^{3})$, while $Z=_{3}0$ works if $\neg(\exists^{2})$ as all functionals on Baire space are continuous then; for the forward implication, consider $(\exists^{2})\vee \neg(\exists^{2})$ and use Corollary \ref{doors} in the former case.  The final equivalence now follows from the first equivalence. 
\end{proof}

\subsection{Heine-Borel compactness}\label{disjunkie}
We discuss the rich world of splittings and disjunctions associated to Heine-Borel compactness as in $\HBU$, which we recall:
\be\tag{$\HBU$}
(\forall \Psi:\R\di \R^{+})(\exists \langle y_{1}, \dots, y_{k}\rangle){(\forall x\in I)}(\exists i\leq k)(x\in I_{y_{i}}^{\Psi}).
\ee
Note that $\HBU_{\c}$ similarly expresses the open-cover compactness of Cantor space.   

\smallskip
\noindent
First of all, we establish a nice disjunction for $\WKL$.  
\begin{thm}\label{hingie}
The system $\RCAo$ proves that 
\be\label{corkukkk}
\WKL\asa[(\exists^{2})\vee \HBU_{\c}] \asa [\X\vee \HBU]\asa [\FF\vee\HBU_{\c}]. 
\ee
for any $\X\in \L_{2}$ such that $\ACA_{0}\di \X\di \WKL_{0}$.  
\end{thm}
\begin{proof}
We prove the first equivalence and note that the other equivalences in \eqref{corkukkk} follow in the same way.  
The reverse implication follows from $\HBU_{\c}\di\WKL$ and $(\exists^{2})\di \ACA_{0}\di \WKL_{0}$.  
For the forward implication, note that all
functionals on $\N^{\N}$ are continuous given $\neg(\exists^{2})$, and hence uniformly continuous on $C$ by $\WKL$.  Hence, all functionals on $C$ have an upper bound, which immediately implies $\HBU_{\c}$.  
The law of excluded middle $(\exists^{2})\vee \neg(\exists^{2})$ finishes the proof.
\end{proof}
As noted in Section \ref{axies!}, the systems $\ACAo$ and $\FIVE^{\omega}$ are conservative extensions of their second-order counterparts.  
However, the $\ECF$-translation leaves $\L_{2}$-sentences unchanged, while translating $(\exists^{2})$ to `$0=1$'.  As a result, the disjuncts in the first equivalence in \eqref{dagcorkukkk} below are independent.  
\begin{cor}\label{hunsruck}
The system $\RCAo$ proves that 
\be\label{dagcorkukkk}
\WKL\asa[\FIVE \vee (\exists^{2})\vee \HBU_{\c}] \asa [\X\vee\FF \vee \HBU].  
\ee
for any $\X\in \L_{2}$ such that $\X\di \WKL_{0}$.  
\end{cor}
Secondly, let $\T_{1}$ be \cite{yukebox}*{Theorem~4.5.2} i.e.\ the $\L_{2}$-sentence: \emph{For all $k\in \N$ and all compact metric spaces $X$ and continuous functions $F:X\di X$,  $F^{k}$ is a continuous function from $X$ into $X$.}
Note that over $\RCA_{0}$, the statement $\T_{1}$ is equivalent to $\WKL\vee \Sigma_{2}^{0}\textsf{-IND}$,  where the latter is the induction schema restricted to $\Sigma_{2}^{0}$-formulas. 
\begin{cor}\label{thedamhasbroken}
The system $\RCAo$ proves 
\be\label{xxxxx}
\T_{1}\asa [ \HBU\vee \Sigma_{2}^{0}\textup{\textsf{-IND}}]\asa  [\WKL\vee\Sigma_{2}^{0}\textup{\textsf{-IND}}]\asa [\FF\vee \HBU\vee \Sigma_{2}^{0}\textup{\textsf{-IND}}].  
\ee
\end{cor}
\begin{proof}
We only need to prove the first equivalence. 
The reverse direction is immediate as $\HBU\di \WKL\di \T_{1}$ and $  \Sigma_{2}^{0}\textup{\textsf{-IND}}\di \T_{1}$. 
For the forward direction, 
\[
\T_{1}\di[ \WKL\vee \Sigma_{2}^{0}\textup{\textsf{-IND}}]\di [\ACA_{0}\vee \HBU\vee \Sigma_{2}^{0}\textup{\textsf{-IND}}], 
\]
since $\ACA_{0}$ implies $\Sigma_{k}^{0}\textup{\textsf{-IND}}$ (for any $k$), and we obtain the  equivalence in \eqref{xxxxx}.

\smallskip

We provide another proof of the forward direction that will be useful for Section~\ref{opinion}.  Assume $\T_{1}$ and consider $\Sigma_{2}^{0}\textup{\textsf{-IND}}\vee \neg[\Sigma_{2}^{0}\textup{\textsf{-IND}}]$.
In the erstwhile case, we are done.  In the latter case, we must have $\WKL$ due to $\T_{1}\asa [\WKL\vee \Sigma_{2}^{0}\textup{\textsf{-IND}}]$; since $(\exists^{2})\di \ACA_{0}\di \Sigma_{2}^{0}\textup{\textsf{-IND}}$, we also obtain $\neg(\exists^{2})$, and hence $\HBU$ as in the proof of the theorem, and we are done. 
\end{proof}
\noindent
Thirdly, while \eqref{corkukkk} and \eqref{xxxxx} may come across as \emph{spielerei}, $\WKL\asa [\ACA_{0}\vee \HBU]$ is actually of great conceptual importance, as follows.  
\begin{tempie}\label{kcuf}\rm
To prove a theorem $\T$ in $\WKL^{\omega}_{0}$, proceed as follows:  
\begin{enumerate}
 \renewcommand{\theenumi}{\alph{enumi}}
\item Prove $\T$ in $\ACA_{0}$ (or even using $\exists^{2}$), which is \emph{much}\footnote{For instance, the functional $\exists^{2}$ uniformly converts between binary-represented reals and reals-as-Cauchy-sequences.  In this way, one need not worry about representations and the associated extensionality like in Definition \ref{keepintireal}.(5).  By the proof of \cite{kohlenbach4}*{Prop.\ 4.7}, $\exists^{2}$ also uniformly converts a continuous function into an RM-code, i.e.\ we may `recycle' proofs in second-order arithmetic.} easier than in $\WKL_{0}$.\label{itema}
\item Prove $\T$ in $\RCAo+\HBU$ using the existing `uniform' proof from the literature based on Cousin's lemma (See e.g.\ \cite{bartle2,bartle3, gormon, thom2, josting, stillebron, botsko}).\label{itemb}
\item Conclude from \eqref{itema} and \eqref{itemb} that $\T$ can be proved in $\WKL^{\omega}_{0}$.
\end{enumerate}
\end{tempie}
Hence, even though the goal of RM is to find the \emph{minimal} axioms needed to prove a theorem, one can nonetheless achieve this goal by (only) using non-minimal axioms.  
We leave it to the reader to ponder how much time and effort could have been (and will be) saved using the previous three steps (for $\WKL$ or other axioms).    
As an exercise, the reader should try to prove \emph{Pincherle's theorem} (\cite{medvet}*{p.\ 97}) via Template \ref{kcuf}, using realisers for the antecedent as in the original \cite{tepelpinch}. 
The former theorem is studied in \cite{dagsamV}, where Template \ref{kcuf} is used frequently. 

\smallskip

Fourth, in \cite{boulanger}*{Theorem 2.28}, an equivalence between $\neg\WKL_{0}\vee \ACA_{0}$ and the following theorem is established: \emph{there is a complete theory with a non-principal type and only finitely many models up to isomorphism}.  The contraposition of the latter, which we shall denote   
$ \T_{0}$ and satisfies $\T_{0}\asa \WKL_{0}+\neg\ACA_{0}$, is described in \cite{boulanger} as \emph{a peculiar but natural statement} about some pre-ordering.  
\begin{cor}\label{puilo}
The system $\RCAo+ \T_{0}$ proves $\HBU$.  
The system $\RCAo$ proves $\T_{0}\asa [\WKL_{0}+ \neg \ACA_{0}] \asa [ \HBU+ \neg\ACA_{0}]$.  
\end{cor}
\begin{proof}
In $\neg\T_{0}\asa [\neg\WKL_{0}\vee \ACA_{0}]$, use \eqref{corkukkk} to replace $\WKL$ by $\ACA_{0}\vee \HBU$, i.e.\ 
\[
\neg\T_{0}\asa\big[ [\neg\ACA_{0}+ \neg \HBU]\vee \ACA_{0}\big]\asa \big[ \underline{[\neg\ACA_{0}\vee  \ACA_{0}]}+[ \ACA_{0}\vee \neg \HBU]\big].
\]
Omitting the underlined formula, the second (and first) part follows.  

\smallskip

We provide another proof of $\T_{0}\di [ \HBU+ \neg\ACA_{0}]$ that will be useful for Section~\ref{opinion}. 
Since $\T_{0}\di [\WKL_{0} +\neg\ACA_{0}]$, we also have $T_{0}\di [\WKL_{0}+\neg(\exists^{2})]$, and $\HBU$ follows as in the proof of the theorem. 
\end{proof}
By the previous, the negation of $\WKL_{0}$ or $\ACA_{0}$ implies axioms of Brouwer's intuitionistic mathematics, i.e.\ strange (as in `non-classical') behaviour is almost guaranteed.  The equivalence involving $\neg\WKL_{0}\vee  \ACA_{0}$ remains surprising.  By contrast, $\T_{0}$ seems fairly normal, relative to e.g.\ $\T_{1}$, by the following result.  
\begin{cor}
The system $\RCA_{0}$ proves $\T_{1}\asa (\T_{0}\vee \Sigma_{2}^{0}\textup{\textsf{-IND}})$, $\WKL_{0}\asa [\ACA_{0}\vee  \T_{0}]$, and $(\T_{0}\vee \T_{1})\asa \T_{1}$.
\end{cor}
\begin{proof}
The second forward implication follows from $\ACA_{0} \vee \neg\ACA_{0}$, while the second reverse implication is immediate.
The first reverse implication is immediate, while the first forward implication follows from:
\[
\T_{1}\di [\WKL_{0} \vee \Sigma_{2}^{0}\textup{\textsf{-IND}}]\di [\ACA_{0}\vee  \T_{0} \vee \Sigma_{2}^{0}\textup{\textsf{-IND}}]\di [  \T_{0} \vee \Sigma_{2}^{0}\textup{\textsf{-IND}}],
\]
since $\ACA_{0}$ proves induction for any arithmetical formula.  The third equivalence follows by considering all cases in the disjunction that is $\T_{1}$.  
\end{proof}
Similar to Corollary \ref{hunsruck}, Theorem \ref{builurhouse} has the following corollary. 
Note that the $\ECF$-translation again implies the independence of the disjuncts in \eqref{dagbayor}, except that we do not know whether $\T_{0}\di \MUC$, over say $\RCAo$. 
\begin{cor}\label{dagbuilurhouse}
The system $\RCAo+\QFAC^{2,0}+\FF$ proves 
\be\label{dagbayor}
 [(\exists^{3})\vee \MUC] \asa (\kappa_{0}^{3}) \asa [(\exists^{3})\vee \MUC \vee \T_{0}].
\ee
\end{cor}
\begin{proof}
The first equivalence and the second forward direction is immediate in light of \eqref{bayor}.  For the second reverse direction, $\T_{0}$ implies $\WKL_{0}$ and $\neg\ACA_{0}$ by definition.  
The latter implies $\neg(\exists^{2})$, i.e.\ all functions on Cantor space are continuous, and the fan functional as in $\FF$ readily yields $(\kappa_{0}^{3})$.  
\end{proof}
Finally, the negation of $\HBU$ also occurs naturally as follows, where we recall:
\be\tag{$\UATR$}
(\exists \Phi^{1\di 1})(\forall X^{1}, f^{1})\big[\WO(X)\di H_{f}(X, \Phi(X,f)) \big], 
\ee
\begin{thm}\label{kukuku}
The system $\RCAo+\FF+\QFAC^{2,1}$ proves $(\exists^{2})\asa [\UATR\vee \neg \HBU]$. 
\end{thm}
\begin{proof}
For the forward implication, consider $\HBU\vee \neg \HBU$.  In the former case, we obtain $\UATR$ by \cite{dagsam}*{Cor.\ 6.6} and \cite{dagsamIII}*{Theorem 3.3}. 
For the reverse implication, $[\neg\HBU +\FF] \di (\exists^{2}) $, which follows from $\MUC\di \HBU$ and Theorem \ref{kwazoeloe}.
\end{proof}
\begin{rem}\label{DFG}\rm
It is a natural RM-question, posed previously by Hirschfeldt (see \cite{montahue}*{\S6.1}), whether the extra axioms are needed in the base theory of Theorem \ref{kukuku}.  
The answer is positive in this case: the $\ECF$-translation converts the equivalence in the theorem to $(0=1)\asa [(0=1)\vee \neg \WKL]$, which is only true if $\WKL$ (which is exactly $[\FF]_{\ECF}$) is in the base theory.  Hence, the base theory needs $\WKL$.  
\end{rem}
The above results, \eqref{corkukkk} and \eqref{xxxxx} in particular, suggests that mathematical naturalness does not inherit to disjuncts, which is in accordance with our intuitions. 

\subsection{Other covering theorems}\label{others}
We study two covering lemmas related to $\HBU$, namely the Lindel\"of lemma and a weak version of $\HBU$. 
\subsubsection{The Lindel\"of lemma}
We study splittings and disjunctions for the \emph{Lindel\"of lemma} $\LIN$ from \cite{dagsamIII}.  
We stress that our formulation of $\HBU$ and $\LIN$ is faithful to the original theorems from 1895 and 1903 by Cousin (\cite{cousin1}) and Lindel\"of (\cite{blindeloef}).  
\bdefi[$\LIN$] 
For every $\Psi:\R\di \R^{+}$, there is a sequence of open intervals $\cup_{n\in \N}(a_{n}, b_{n})$ covering $\R$ such that $(\forall n \in\N)(\exists x \in \R)[(a_{n}, b_{n}) = I_{x}^{\Psi} ]$.  
\edefi
The final result in the following theorem should be compared to \eqref{corkukkk}.
\begin{thm} \label{linnenkes}
Let $\X\in \L_{2}$ be such that $\ACA_{0}\di \X\di \WKL_{0} $.  
\begin{enumerate}
 \renewcommand{\theenumi}{\alph{enumi}}
\item The system $\RCAo+\QFAC^{0,1}$ proves $\LIN\asa [\HBU \vee \neg\WKL]\asa [\HBU\vee\neg \X]$.   
\item If $\Y\in \L_{2}$ is provable in $\ACA_{0}$ but not in $\RCA_{0}$, then $\RCAo$ proves $\Y\vee \LIN$, as well as $\FIVE\vee (\exists^{2})\vee \LIN$.
\end{enumerate}
\end{thm}
\begin{proof}
For the first item, $\RCAo+\QFAC^{0,1}$ proves $[\LIN+\WKL]\asa \HBU$ by \cite{dagsamIII}*{Theorem 3.13}.  Hence, the first forward implication follows from $\WKL\vee \neg\WKL$. 
For the first reverse implication, $\LIN$ follows from $\HBU$ by the aforementioned equivalence.  In case $\neg\WKL$ holds, we also have $\neg(\exists^{2})$, as $(\exists)^{2}\di \WKL$.  
Hence, all functionals on $\R$ are continuous by \cite{kohlenbach2}*{Prop.\ 3.12}, and the countable sub-cover provided by the rationals suffices for the conclusion of $\LIN$.   The second 
equivalence follows in the same way by considering $\X\vee \neg \X$.  
For the second item, consider $(\exists^{2})\vee \neg(\exists^{2})$.  
\end{proof}
We now obtain a nice corollary to Theorems \ref{kukuku} and \ref{linnenkes}. In light of Remark~\ref{DFG}, $\WKL$ also suffices for the base theory in the latter theorem.  
\begin{cor}\label{kukuku27}
The system $\WKL_{0}^{\omega}+\QFAC^{2,1}$ proves $(\exists^{2})\asa [\UATR\vee \neg \HBU]$.  
The system $\RCAo+\QFAC^{2,1}$ proves $(\exists^{2}) \asa [\UATR\vee \neg\LIN]$.
\end{cor}
\begin{proof}
The first forward implication follows as in the proof of Theorem \ref{kukuku}.  For the first reverse implication, the case $\neg \HBU$ implies $(\exists^{2})$ by considering \eqref{hingie}.  
 The second equivalence now follows from the first item of Theorem \ref{linnenkes}.
\end{proof}
\noindent
Finally, we discuss foundational implications of our results.  Now, \eqref{corkukkk} implies:  
\be\label{dddf}
\neg \LIN\asa [\WKL+\neg\HBU] \textup{ and } \neg\LIN \di (\exists^{2}).  
\ee
On one hand, thanks to the $\ECF$-translation, $\WKL_{0}^{\omega}+\HBU$ is a conservative extension of $\WKL_{0}$, which in turn is a $\Pi_{2}^{0}$-conservative extension of primitive recursive arithmetic \textsf{PRA}.  
The latter is generally believed to correspond to Hilbert's \emph{finitistic mathematics} (\cite{tait1}).      
Hence, following Simpson's remarks on finitistic mathematics (\cite{simpson2}*{IX.3.18}), $\WKL_{0}^{\omega}+\HBU$ also contributes to the partial realisation of Hilbert's program for the foundations of mathematics. 
On the other hand, $\RCAo+\WKL+\neg\HBU$ and $\RCAo+\QFAC^{0,1}+\neg\LIN$ imply $(\exists^{2})$, i.e.\ these systems do not contribute to Hilbert's program in the aforementioned way.  

\smallskip

Hence, if one values partial realisations of Hilbert program (which are called `very important' by Simpson in \cite{simpson2}*{IX.3.18}), then $\HBU$ and $\LIN$ are practically forced upon one, in light of the previous.   
However, these covering lemmas require full second-order arithmetic as in $\Z_{2}^{\Omega}$ for a proof, i.e.\ they fall \emph{far} outside of the \emph{Big Five} classification of RM.  

\smallskip
Finally, we consider the Lindel\"of lemma for Baire space, denoted $\LIN(\N^{\N})$ and studied in \cite{dagsamIII, dagsamV}.
Similar to Corollary \ref{hunsruck}, one can prove the following equivalence:
\be\label{hoerah}
[\WKL\vee \LIN(\N^{\N})]\asa[\FIVE \vee (\exists \Xi)\LIN(\Xi) \vee (\exists^{2})\vee \HBU_{\c}] 
\ee
where $(\exists \Xi)\LIN(\Xi)$ states the existence of a functional $\Xi^{2\di (0\di 1)}$ that outputs the countable sub-cover from $\LIN(\N^{\N})$.

\subsubsection{Weak Heine-Borel compactness}\label{WHBU}
We study $\WHBU$, a weak version of $\HBU$ based on \emph{weak weak K\"onig's lemma} ($\WWKL$ hereafter; see \cite{simpson2}*{X.1}).  
Note that $\WWKL$ is exceptional in that it is a theorem from the RM zoo that does sport a number of equivalences involving natural/mathematical statements.    

\smallskip

In particular, by \cite{simpson2}*{X.1.9}, $\WWKL$ is equivalent to the statement that any cover $\cup_{n\in \N}(a_{n}, b_{n})\subset [0,1]$ is such that $\sum_{n=0}^{\infty}|a_{n}-b_{n}|\geq 1$, which is of independent\footnote{It is an interesting historical tidbit that a two-dimensional version of \cite{simpson2}*{X.1.9.3} was Borel's motivation for formulating and proving the (countable) Heine-Borel theorem (\cite{opborrelen}*{p.\ 50, Note}).} historical interest.  
We define the higher-order version of this covering theorem as:
\be\tag{$\WHBU$}\textstyle
(\forall \Psi:\R\di \R^{+}, k\in \N)(\exists \langle y_{1}, \dots, y_{n}\rangle)\big( \sum_{i=1}^{n}  |I_{y_{i}}^{\Psi}|\geq 1-\frac{1}{2^{k}}   \big).
\ee
We could also use the statement $\HBU_{\textsf{ml}}$ from \cite{samcie18}*{\S3.3} instead of $\WHBU$, but the latter is more elegant, and does not depend on 
the notion of randomness.  
\begin{thm} Let $\X\in \L_{2}$ be such that $\ACA_{0}\di \X\di \WWKL_{0} $.  
$\RCAo$ proves
\be\label{ohjee}
\WWKL\asa [(\exists^{2})\vee \WHBU]\asa [\X\vee \WHBU].  
\ee
\end{thm}
\begin{proof}
This theorem is proved in the same way as Theorem \ref{hingie}.  
Indeed, for the first forward implication, consider $(\exists^{2})\vee \neg(\exists^{2})$ and note that in the latter case $\cup_{q\in [0,1]\cap \Q}I_{q}^{\Psi}$ is a countable sub-cover of the canonical cover since all functions are continuous.  
The first reverse implication is trivial in light of \cite{simpson2}*{X.1.9}, and the other equivalences are proved similarly (and using Theorem \ref{hingie}). 
\end{proof}
Similar to Corollary \ref{hunsruck}, we have the following corollary.  
\begin{thm} Let $\X\in \L_{2}$ be such that $\X\di \WWKL_{0} $.  
$\RCAo$ proves
\be\label{dagohjee}
\WWKL\asa [\FIVE\vee (\exists^{2})\vee \WHBU]\asa [\X\vee\FF \vee  \WHBU].  
\ee
\end{thm}
\noindent
One also readily proves that \eqref{hoerah} can be extended to $\eqref{hoerah}\asa \WWKL\vee \LIN(\N^{\N})$.

\smallskip

The following version of Corollary \ref{thedamhasbroken} for $\WHBU$ is readily proved based on $\WWKL\vee\Sigma_{2}^{0}\textup{\textsf{-IND}}$ and \eqref{ohjee}.  
We can prove similar results for the \emph{strong bounding principles} and \emph{bounded comprehension principles} instead of induction (\cite{simpson2}*{p.\ 72}).   
\begin{cor} The system $\RCAo$ proves 
\[
 [ \WHBU\vee \Sigma_{2}^{0}\textup{\textsf{-IND}}  ] \asa  [ \WWKL\vee \Sigma_{2}^{0}\textup{\textsf{-IND}} ].  
\]
\end{cor}
\noindent
We can also obtain a version of Theorem \ref{linnenkes} for $\WHBU$. 
\begin{cor}
 Let $\X\in \L_{2}$ be such that $\ACA_{0}\di \X\di \WWKL_{0} $; the system $\RCAo+\QFAC^{0,1}$ proves $\LIN\asa [\WHBU \vee \neg\WWKL]\asa [\WHBU\vee\neg \X]$.   
\end{cor}
Finally, let $(n+1)$-$\WWKL$ be the generalisation of $\WWKL$ to trees computable in the $n$-th Turing jump, as formulated in \cite{avi1337}.
Note that $\ACA_{0}\di (n+2)\textsf{-}\WWKL\di  (n+1)\textsf{-}\WWKL\not\di \WKL$ over $\RCA_{0}$.   While \eqref{ohjee} applies to $\X\equiv (n+1)$-$\WWKL$, we also have the following corollary. 
\begin{cor} For $n\geq 1$, $\RCAo$ proves  $[ \HBU\vee n\textsf{-}\WWKL  ] \asa  [ \WKL\vee n\textsf{-}\WWKL ]$.   
\end{cor}
\begin{proof}
The forward implication is immediate, while for the reverse implication follows from \eqref{corkukkk} as $\WKL\di [\ACA_{0}\vee \HBU]\di [n\textsf{-}\WWKL\vee \HBU]$.
\end{proof}
Finally, it is a natural question if there are other theorems in the RM zoo for which we can find results like \eqref{corkukkk} and \eqref{ohjee}.  
We will provide a positive answer for (fragments of) \emph{Ramsey's theorem} in a future publication.  

\section{Foundational musings}\label{opinion}
We provide an explanation for our results regarding splittings and disjunctions in Section \ref{diskol}, following an introduction in Section \ref{fintro}.
The bigger picture is discussed in Sections \ref{skylar} and \ref{diskol2}.
\subsection{Introduction: continuity and discontinuity}\label{fintro}
%
%
%
In the below discussion, a central role is played by \emph{continuity}.  
To be absolutely clear, our use of `continuity' refers to the usual `epsilon-delta' definition of functionals of type two or higher, unless stated otherwise.
By \cite{kohlenbach2}*{Prop.\ 3.7 and~3.12}, the existence of a \emph{discontinuous}, i.e.\ not everywhere continuous, functional is equivalent to $(\exists^{2})$, for both $\N^{\N}$ and $\R$. 

\smallskip

On one hand, it is well-known that $\L_{2}$ provides \emph{representations} for discontinuous functions.  
For instance, measurable functions are represented by sequences of continuous functions in RM (see \cite{simpson2}*{X.1.11}).  Furthermore, the basic theory of Borel functions and analytic sets can be developed in $\ATR_{0}$ \emph{via codes for Borel sets} (see \cite{simpson2}*{V}).  In other words, second-order arithmetic can \emph{model/represent} discontinuous phenomena, and the purpose of this approach is to reconstitute these concepts in a way that accommodates their main applications within regions of the G\"odel hierarchy (see \cite{sigohi}) that are as far down in the hierarchy as possible.  For this kind of purpose, subsystems of $\Z_{2}$ have been tremendously useful.

\smallskip

On the other, for $A\in \L_{2}$ provable in $\Z_{2}$, $\RCAo+A$ cannot prove the existence of a discontinuous function, thanks\footnote{The $\ECF$-translation of $(\exists^{2})$ is `$0=1$', while it does not change $A\in \L_{2}$; see Remark \ref{ECFrem}.} to the $\ECF$-translation and assuming $\Z_{2}$ is consistent. 
Hence, second-order arithmetic can model/represent (certain) discontinuous phenomena, but it cannot prove (in the extended language $\L_{\omega}$) the existence of e.g.\ a discontinuous function like $\exists^{2}$.  
Nonetheless, discontinuous functions entered the mathematical mainstream already around 1850 thanks to Riemann's \emph{Habilschrift}:
\begin{quote}
Riemann’s work may be said to mark the beginning of a theory of the mathematically discontinuous, although there are isolated examples in Fourier’s and Dirichlet’s works. It planted the discontinuous firmly upon the mathematical scene. (\cite{kleine}*{p.\ 116})
\end{quote}
Thus, discontinuous functions are part of ordinary, i.e.\ non-set-theoretical, mathematics, predating the earliest days of set theory.    
Discontinuous phenomena can be \emph{modeled} or \emph{represented} in second-order arithmetic, but the latter cannot prove the existence of the underlying phenomena, even for the most basic case of a discontinuous function on $\R$.
As we will see the next section, the aforementioned limitation of $\L_{2}$ plays an important role in the study of splittings and disjunctions.  

\subsection{Discontinuity: the genesis of splittings and disjunctions}\label{diskol}
By the above, higher-order RM features lots of examples of splittings and disjunctions, esp.\ compared to second-order RM.
We now provide a possible explanation for this observation, i.e.\ we answer the question \emph{why} there are so many splittings and disjunctions in higher-order RM, compared to second-order RM. 

\smallskip

First of all, it goes without saying that the language of higher-order RM  is much richer than the language of second-order arithmetic.  
Hence, more mathematics can be expressed in $\L_{\omega}$, but this observation \emph{alone} does not provide a satisfying explanation.  
The essential observation is that, as discussed in Section~\ref{fintro}, second-order arithmetic cannot directly accommodate discontinuous phenomena, while higher-order arithmetic of course can.
Thus, sentences of $\L_{\omega}$ can be divided in the following three natural categories:
\begin{enumerate}
 \renewcommand{\theenumi}{\alph{enumi}}
\item Sentences implying the existence of discontinuous\footnote{While $\exists^{2}$ is the `archetype' of a discontinuous function, Gandy's `superjump', introduced in \cite{supergandy}, has a characterisation in terms of discontinuities, as discussed in \cite{hartjeS}.} objects.\label{kart}
\item Sentences implying the continuity of a certain class of objects.\label{ley}
\item Sentences that are `neutral' regarding continuity, i.e.\ consistent with all sentences from items \eqref{kart} and \eqref{ley}\label{ley2}.
\end{enumerate}
Items \eqref{kart} and \eqref{ley} are often connected: while $\exists^{2}$ is discontinuous, $\neg(\exists^{2})$ implies that all functions on $\R$ are continuous by \cite{kohlenbach2}*{Prop.\ 3.12}.
In light of the results in \cite{ishicont}, the connection between items \eqref{kart} and \eqref{ley} even exists in constructive mathematics.
Moreover, since it implies $\neg(\exists^{2})$, $\neg \WKL$ belongs to item \eqref{ley}, while $\WKL$ belongs to item \eqref{ley2}.  Thus, $\L_{2}$ is not restricted to \eqref{ley2}, but $\L_{2}$ 
just lacks the expressiveness to state the `logical consequences' of $\neg\WKL$, namely $\neg(\exists^{2})$ and that all functions on $\R$ are continuous.  
To fully appreciate the previous, the reader should now consult the final parts of the proofs of Corollaries~\ref{thedamhasbroken} and~\ref{puilo}.

\smallskip

Secondly, many of the results in the previous sections can be viewed as splitting sentences of $\L_{\omega}$ in weaker (and generally independent) components from items \eqref{kart} or \eqref{ley}, and \eqref{ley2}.
For instance, the trichotomy in the previous paragraph suggests the following way of obtaining splittings: from a sentence $A$ from item \eqref{kart} (resp.\ item \eqref{ley}), 
derive a weaker sentence $B$ expressing some discontinuity (resp.\ continuity) property, and a weaker sentence $C$ from item \eqref{ley2}; $B$ and $C$ should be strong enough to guarantee $A\asa [B+C]$.  
This kind of splitting is obtained in Sections \ref{CoC}-\ref{more}. 

\smallskip

Third, a different but related approach is based on the law of excluded middle, $(\exists^{2})\vee \neg(\exists^{2})$ in particular; other instances are used as well, but the idea is the same, as follows:   
starting from a relatively weak principle $C$, $(\exists^{2})\vee \neg(\exists^{2})$ implies $D\vee E$, where $D, E$ are (much) stronger than $C$. 
In particular, in case $\neg(\exists^{2})$, all functions on the reals are continuous, and uncountable covers then reduce to countable ones.  Hence, the Lindel\"of lemma becomes trivial, while Heine-Borel compactness as in $\HBU$ simply follows from $\WKL$, and $\MUC$ follows from $\FF$, i.e.\ we obtain a (hard to prove) sentence from item \eqref{ley2}, or a sentence from item \eqref{ley}.  
In case $(\exists^{2})$, we are obviously in item \eqref{kart}.  Since $D$ and $E$ both imply $C$, we obtain $C\asa [D \vee E]$.  This kind of disjunction is obtained in Sections \ref{more}-\ref{others}.

\smallskip

In conclusion, $\L_{\omega}$ can represent discontinuous objects \emph{directly}, in contrast to the indirect approach provided by $\L_{2}$.  This particular richness gives rise to the trichotomy above in items \eqref{kart}-\eqref{ley2}.  
Along the lines of the latter, one can obtain plenty of splittings and disjunctions in $\L_{\omega}$, based on the previous two paragraphs.  In other words, the lack of splittings and disjunctions in classical RM 
is due to the weak expressive power of $\L_{2}$, in particular the fact that it can only model, but not prove, discontinuous phenomena.        

\smallskip

Finally, splittings and disjunctions are not the only example of the impact of the limitations of $\L_{2}$.  The following remark presents another one.   
\begin{rem}[Explosions in higher-order arithmetic]\label{sexplosion}\rm
We say that two sentences $A, B$ cause an `explosion' if $A+B$ is much stronger than $A$ or $B$ considered separately (say in $\RCAo$).  We show in this remark that there are natural explosions in $\L_{\omega}$, which disappear in the absence of discontinuous functions.  

\smallskip

First of all, both $(\exists^{2})$ and $\HBU$ are weak \emph{in isolation}, i.e.\ conservative over $\ACA_{0}$, but the combination implies $\ATR_{0}$ by \cite{dagsam}*{\S6}; see also \cite{dagsamIII}*{\S3}.
However, $\ACA_{0}+\HBU$ is conservative\footnote{The $\ECF$-interpretation translates $\HBU$ to $\WKL$, and the latter follows from $\ACA_{0}$.} over $\ACA_{0}$.   

\smallskip

Secondly, $(\exists^{2})$ and the \emph{Lindel\"of lemma for $\N^{\N}$}, denoted $\textsf{LIND}(\N^{\N})$ in \cite{dagsamV}, are weak \emph{in isolation}, i.e.\ conservative over $\ACA_{0}$, but the combination implies $\FIVE$ by \cite{dagsamV}*{\S5}.  
However, $\ACA_{0}+\textsf{LIND}(\N^{\N})$ is conservative\footnote{The $\ECF$-interpretation translates all versions of the Lindel\"of lemma to trivialities.} over $\ACA_{0}$.   

\smallskip

The previous two explosions show that the presence of discontinuous functions has a great impact on the logical strength of (uncountable) covering theorems.  
\end{rem}

\subsection{To be or not to be continuous}\label{skylar}
The results in Section \ref{main} and \cite{dagsamIII, dagsamV, dagsam, dagsamII} identify huge differences between second- and higher-order RM.  
As discussed in the previous section, these results trace back to the fact that higher-order (resp.\ second-order) arithmetic can (resp.\ cannot) \emph{directly} represent discontinuous phenomena.  
Hence, the question arises whether one should adopt the higher-order framework instead of second-order arithmetic for the formalisation of mathematics.  

\smallskip

In this section, we argue that one must adopt the higher-order framework \emph{in either of the following situations}: 
\begin{enumerate}
 \renewcommand{\theenumi}{\alph{enumi}}
\item {if one wants to formalise mathematics in a way close to the original},\label{trunk}
\item  the second-order formalisation should be faithful \emph{in scope} to the original.\label{barf}
\end{enumerate}
As we will see, the caveat in item \eqref{barf} can be summarised as \emph{faithfulness is hard}.  
To be absolutely clear, `faithful' means that the second-order formalisation has the same scope or generality as the original, i.e.\ we are not implying that the formalisation `should look (exactly) like the original'.
The caveat in item \eqref{trunk} does discuss this idea of `close to the original', and is actually inspired by the development of the \emph{gauge integral}, which we discuss first, as follows. 

\smallskip

The gauge integral is a generalisation of the Lebesgue and (improper) Riemann integral;
this integral was introduced by Denjoy (\cite{ohjoy}), in a different and more complicated form, around the same time as the Lebesgue integral; the reformulation of Denjoy's integral by Henstock and Kurzweil in Riemann-esque terms (See \cite{bartle1337}*{p.~15}), provides a \emph{direct} and elegant formalisation of the \emph{Feynman path integral} (\cites{burkdegardener,mullingitover,secondmulling}) and financial mathematics (\cites{mulkerror, secondmulling}).    
In a nutshell, the gauge integral is just the Riemann integral with the constant `$\delta\in \R^{+}$' in the usual $\eps$-$\delta$-definition replaced by a \emph{function} $\delta:\R\di \R^{+}$, a small but significant change. 

\smallskip

Now, the first step in the development of the gauge integral is always to show that this integral is well-defined, using the \emph{Cousin lemma}, which implies $\HBU$.
As shown in \cite{dagsamIII}*{\S3}, $(\exists^{2})$ and $\HBU$ are essential for the development of the \emph{gauge integral} (\cite{bartle1337}) in that the former are equivalent to various basic properties of the gauge integral.  
Furthermore, Cousin's lemma from \cite{cousin1}*{p.\ 22} dates back\footnote{
The collected works of Pincherle contain a footnote by the editors (See \cite{tepelpinch}*{p.\ 67}) which states that the associated \emph{Teorema} (published in 1882) corresponds to the Heine-Borel theorem.  Moreover, Weierstrass proves the Heine-Borel theorem (without explicitly formulating it) in 1880 in \cite{amaimennewekker}*{p.\ 204}.   A detailed motivation for these claims may be found in \cite{medvet}*{p. 96-97}.
} about 135 years.  
Thus, $(\exists^{2})$ and $\HBU$ should count as `core' or `ordinary'  mathematics.  

\smallskip

The previous observations will give rise to different reactions in different people: one person will see the above as a convincing argument for the adoption of higher-order arithmetic, while another person will see this as another subject that needs to be formalised in $\L_{2}$.  
To avoid a deadlock, we recall the connection between physics and the gauge integral from \cite{dagsamIII}*{\S3.3} as follows: Muldowney has expressed the following opinion in a private communication.  
\begin{quote}
There are a number of different approaches to the formalisation of Feynman’s path integral.  However, 
if one requires the formalisation to be close to Feynman’s original formulation, then the gauge integral
is really the only approach.
\end{quote}
Arguments for this opinion, including major contributions to Rota's program for the Feyman integral, may be found in \cite{mully}*{\S A.2}.  
We adopt a similar stance regarding the adoption of higher-order arithmetic:  anyone interested in a \emph{direct}\footnote{In both second- and higher-order RM, real numbers are represented by Cauchy sequences, but the associated practice is actually close to mathematical practice, as discussed in Remark \ref{forealsteve}.} logical formalisation of the gauge integral, has no choice but to adopt the higher-order framework.  
In other words, \emph{assuming one wants to formalise the gauge integral in a way close to the original}, one is wedded to $(\exists^{2})$ and $\HBU$.  To be clear, this does not exclude the possibility 
of alternative formalisations in $\L_{2}$, at the cost of a development that is (very) different from the literature.  Nonetheless, the treatment in \cite{walshout} is ultimately based on fundamental results of the gauge integral from \cite{zwette}.

\smallskip

We now turn to item \eqref{barf}, introduced at the beginning of this section and summarised as \emph{faithfulness is hard}.   
First of all, we provide an example where it is \emph{easy} (in terms of logical strength) to show that the second-order formalisation in RM is faithful \emph{in scope and generality} to the original.  
\begin{exa}[Coding continuous functions]\label{exak}\rm
As is well-known, continuous functions are represented by codes in RM (see \cite{simpson2}*{II.6.1}).  
It is then a natural question whether codes actually capture \emph{all} continuous functions (say in a weak system).  Indeed, if codes only captured a special sub-class, then 
a theorem of RM would be about that sub-class, and not about all continuous functions.  However, Kohlenbach has shown in \cite{kohlenbach4}*{\S4} that $\WKL$ suffices to prove that   
every continuous function has a code.  Hence, the RM of $\WKL$ does not really change if we introduce codes, i.e.\ there is a perfect match between the theorems in second- and higher-order arithmetic.    
Thus, \emph{second-order} $\WKL$ (working in $\RCAo$) proves that the the second-order formalisation is faithful \emph{in scope} to the original.
\end{exa}
Secondly, we provide an example where it is \emph{extremely hard} to show that the second-order formalisation is faithful in scope to the original.
\begin{exa}[Coding measurable functions]\label{xle}\rm
Measurable functions are represented in RM by sequences of codes for continuous functions (see \cite{simpson2}*{X.1.11}).
As in Example~\ref{exak}, it is a natural question whether codes actually capture \emph{all} measurable functions (again in a weak system).  Indeed, if codes only captured a special sub-class, then 
a theorem of RM would be about that sub-class, and not about all measurable functions.  Now, \emph{Lusin's theorem} (see e.g.\ \cite{toames}*{1.3.28}) guarantees that every measurable function can be approximated 
by a sequence of continuous functions.  However, as shown in \cite{dagsamVI}, Lusin's theorem (and the same for many similar approximation theorems) implies $\WHBU$ from Section \ref{WHBU}, and the latter is not provable in $\SIXK$ for any $k$, i.e.\ $(\exists^{3})$ is required as for $\HBU$ (see also \cites{dagsamV}).  
\end{exa}
In light of Example \ref{xle}, to guarantee that theorems about codes for measurable functions have the same generality as theorems about measurable functions, i.e.\ to show that the second-order formalisation is faithful to the original, we require $\WHBU$, a \emph{third-order} theorem only provable in full second-order arithmetic $\Z_{2}^{\Omega}$.

\smallskip
\noindent
Finally, while item \eqref{trunk} can be dismissed as an aesthetic preference, dismissing item~\eqref{barf} as unimportant betrays a certain formalist view of the foundations of mathematics. 
We finish this section with a quote on the adequacy of $\L_{2}$.
\begin{quote}
We focus on the language of second order arithmetic, because that language is the
weakest one that is rich enough to express and develop the bulk of core mathematics. (\cite{simpson2}*{Preface})
\end{quote}
In conclusion, we believe Simpson's claim is \textbf{wrong} \emph{in the situations described by items \eqref{trunk} and \eqref{barf}} above.  Indeed, mathematics is (and has been for a long time) replete with discontinuous phenomena and, in our opinion, indirectly dealing with the latter via codes is not satisfactory as this obfuscates a number of interesting \emph{mathematical}\footnote{In the basic development of the gauge integral (\cite{zwette}), to show that the latter is well-defined, one applies the Cousin lemma (and hence $\HBU$) to the canonical cover associated to the gauge function; the latter is continuous if and only if the original function is Riemann integrable.  In other words, viewing the gauge integral as an extension of the Riemann integral, one essentially always works with uncountable covers generated by discontinuous functions.  Thus, the first explosion in Remark~\ref{sexplosion} is quite natural from this (mathematical) point of view.} phenomena, like the plethora 
of splittings and disjunctions from Section \ref{main} and the `explosions' from Remark \ref{sexplosion}.  Moreover, the requirement that e.g.\ codes capture all measurable functions  is based on Lusin's theorem and hence $\WHBU$, only provable in full second-order arithmetic $\Z_{2}^{\Omega}$.  In this light, one might as well work directly in higher-order arithmetic. 

\subsection{The bigger picture}\label{diskol2}
We discuss the place occupied by higher-order RM in the grand scheme of things, esp.\ how higher-order arithmetic relates to fields based on second-order arithmetic, like RM and (classical) computability theory.
We start with some historical considerations, leading up to our conclusion. 

\smallskip

If the history of (the foundations of) mathematics teaches us anything, it is that foundational topics can be quite emotionally charged.  
Let us therefore start with a clear caveat: there is nothing \emph{wrong} with second-order arithmetic, RM and its coding, or classical computability theory.  
These are extremely interesting and equally successful enterprises, and perhaps therein lies the nature of the \emph{issue} we wish to discuss in this section, as follows. 

\smallskip 

The aforementioned \emph{issue} has a proud ancestry, and discussing an example will hopefully clarify things.  
The issue at hand is that successful theories (models/techniques/\dots) that go unchallenged for a long time develop an air of being mostly finished or complete, i.e.\ the grand underlying principles are know, and the rest is simple refinement.  
For instance, in the case of late 19th century physics, the following quote from the Nobel-prize winner Michelson is telling:
\begin{quote}
While it is
never safe to affirm that the future of Physical Science has no
marvels in store even more astonishing than those of the past,
it seems probable that most of the grand underlying principles
have been firmly established and that further advances are to be
sought chiefly in the rigorous application of these principles to
all the phenomena which come under our notice.  (see \cites{mich1, mich2, mich3})
\end{quote}
Weinberg discusses this topic in \cite{wijnberg} and sorts myth from fact, recounting quotes from Planck and Millikan that back Michelson's view.
It is a matter of the historical record that only a couple of decades after Michelson's quote, modern physics was developed, yielding an entire array of new `grand underlying principles'.  

\smallskip

Coming back to mathematics, we believe that the history of second-order arithmetic and associated fields like RM and (classical) recursion theory has been similar: 
this development was extremely successful and impressive, leading to a feeling that the grand underlying principles
had been firmly established.  Indeed, the \emph{G\"odel hierarchy} is a collection of logical systems ordered via consistency strength, or essentially equivalent: ordered via inclusion\footnote{Simpson and Friedman claim that inclusion and consistency strength yield the same G\"odel hierarchy as depicted in \cite{sigohi}*{Table 1} with the caveat that e.g.\ $\RCA_{0}$ and $\WKL_{0}$ have the same first-order strength, but the latter is strictly stronger than the former.\label{fooker}}.  This hierarchy is claimed to capture most systems that are natural or have foundational import, as follows. 
\begin{quote}
{It is striking that a great many foundational theories are linearly ordered by $<$. Of course it is possible to construct pairs of artificial theories which are incomparable under $<$. However, this is not the case for the ``natural'' or non-artificial theories which are usually regarded as significant in the foundations of mathematics.} (\cite{sigohi})
\end{quote}
Burgess and Koellner corroborate this claim in \cite{dontfixwhatistoobroken}*{\S1.5} and \cite{peterpeter}*{\S1.1}.
The G\"odel hierarchy is a central object of study in mathematical logic, as e.g.\ argued by Simpson in \cite{sigohi}*{p.\ 112} or Burgess in \cite{dontfixwhatistoobroken}*{p.\ 40}.  
Precursors to the G\"odel hierarchy may be found in the work of Wang (\cite{wangjoke}) and Bernays (see \cite{theotherguy}, and the translation in \cite{puben}).
Friedman (\cite{friedber}) studies the linear nature of the G\"odel hierarchy in detail.  

\smallskip

In contrast to the aforementioned\footnote{Simpson's above grand claim notwithstanding, there are some examples of theorems (predating $\HBU$ and \cite{dagsamIII}) that also fall outside of the G\"odel hierarchy (based on inclusion), like \emph{special cases} of Ramsey's theorem and the axiom of determinacy from set theory (\cites{dsliceke, shoma}).    
} `received view', and starting with the results in \cite{dagsamIII,dagsamV}, a \emph{large} number of \emph{natural} theorems (of higher-order arithmetic) have been identified forming a branch \emph{independent} of the medium range of the G\"odel hierarchy (based on inclusion$^{\ref{fooker}}$).  
Results pertaining to `uniform' theorems are in \cite{dagsamV}, while the results pertaining to $\HBU$ and the gauge integral are in \cite{dagsamIII}.  
We draw the following conclusions from these observations.
\begin{enumerate}
\item Stepping outside $\L_{2}$, as motivated in Section \ref{skylar}, yields a picture completely different from the G\"odel hierarchy.  This linear order is an artifact of the `absence of discontinuity' discussed in Sections \ref{fintro} and \ref{diskol}. 
\item Notions of continuity and discontinuity successful in first- and second-order arithmetic have to be rethought entirely, or abandoned for new notions, to penetrate structures in higher types, and that this remains for the future as mathematics inevitably evolves.
\item By Example \ref{xle}, one needs to accept hard-to-prove theorems of higher-order arithmetic to guarantee that the associated second-order formalisation is faithful.  While the latter enterprise is therefore no less interesting, Simpson's claim pertaining to the adequacy of $\L_{2}$ become untenable. 
\item We conjecture the existence of other branches, independent of both the G\"odel hiearchy and the branch populated by $\HBU$ and its kin.  
\end{enumerate}
%
%

\smallskip

Finally, the reader should read nothing but simple analogy in the above observations: the discovery of modern physics does not compare in any way to recent discoveries in higher-order arithmetic.   
%
%


%
%


\section{Conclusion}\label{konkelfoes}
The following table summarises some of our results, without mentioning the base theory; the latter is generally conservative over $\WKL_{0}$ (or is weaker).  
In light of this, we may conclude that the higher-order framework yields plenty of equivalences for disjunctions and splittings, in contrast to the second-order framework, and this for the reasons discussed in Section \ref{diskol}.  
\begin{figure}[h]
\resizebox{\linewidth}{!}{%
\begin{tabular}{ |c|c|c| } 
 \hline
 $\MUC\asa [\WKL +(\kappa_{0}^{3})+\neg(\exists^{2})]$ & $(\exists^{3})\asa [(Z^{3})+ (\exists^{2})]$ &$(\kappa_{0}^{3})\asa [(Z^{3})+\FF]$ \\ 
  $\MUC\asa [\WKL +(\kappa_{0}^{3})+\neg(S^{2})]$&$(\exists^{3})\asa [(\kappa_{0}^{3})+(\exists^{2})]$ & $[(\kappa_{0}^{3})+\WKL]\asa  [(\exists^{3})\vee \MUC]$  \\
    $\MUC\asa [\WKL +(\kappa_{0}^{3})+\neg(\exists^{3})]$& $(\exists^{3})\asa [\FF+(Z^{3})+\neg\MUC] $& $\FF\asa [(\exists^{2})\vee \MUC]$  \\
    $\MUC\asa [\FF+ \neg(\exists^{2})]$ &  $(\exists^{2})\asa [\UATR\vee\neg\HBU]$ & $\FF\asa [(\exists^{2})\vee (\kappa_{0}^{3})] $\\
    $\MUC\asa [\FF+(Z^{3})+ \neg(S^{2})]$&$(\exists^{2})\asa [\FF+\neg\MUC]$& $(Z^{3})\asa [(\exists^{3})\vee \neg(\exists^{2})]$ \\
        $\MUC\asa [\FF+(Z^{3}) +\neg(\exists^{3})]$&$ \WKL\asa [(\exists^{2})\vee \HBU] $ &$(Z^{3})\asa [(\exists^{3})\vee \neg\FF\vee \MUC]$\\
  $\T_{1}\asa [ \T_{0}\vee \Sigma_{2}^{0}\textup{\textsf{-IND}}]$      &$\WWKL\asa [(\exists^{2})\vee \WHBU]$& $\LIN\asa [\HBU \vee \neg\WKL]$ \\
        \hline
\end{tabular}
}
\caption{Summary of our results}\label{fagure}
\end{figure}\\
Finally, Simpson describes the `mathematical naturalness' of logical systems as:
\begin{quote}
From the above it is clear that the [Big Five] five basic systems $\RCA_{0}$, $\WKL_{0}$, $\ACA_{0}$, $\ATR_{0}$, $\FIVE$ arise naturally from investigations of the Main Question. The proof that these systems are mathematically natural is provided by Reverse Mathematics. (\cite{simpson2}*{I.12})
\end{quote}
We leave it to the reader to decide if the aforementioned results bestow naturalness onto the theorems involved in the equivalences.  
We do wish to point out that some of the theorems in Figure \ref{fagure} are natural, well-established, and date back more than a century already; see Section \ref{skylar} for details.

\begin{ack}\rm
My research was supported by the John Templeton Foundation (grant ID 60842), the Alexander von Humboldt Foundation, and LMU Munich (via the Excellence Initiative and the Center for Advanced Studies of LMU).  
I express my gratitude towards these institutions.  Opinions expressed in this paper do not necessarily reflect those of the John Templeton Foundation.    

\smallskip

The research leading to this paper grew out of my joint project with Dag Normann, the papers \cite{dagsamV,dagsamIII} in particular. 
I thank Dag Normann for his valuable advice, especially regarding the properties of $(Z^{3})$.  
I also thank Denis Hirschfeldt for his valuable suggestions regarding $\T_{0}$.   I thank the anonymous referee for various helpful suggestions, esp.\ pertaining to Section~\ref{opinion}.   
Finally, I thank Anil Nerode, Denis Hirschfeldt, and Steve Simpson for their help shaping Section \ref{opinion}.
\end{ack}

\bye